\documentclass[11pt]{article}

\usepackage{amsmath, amsfonts, amssymb, amsthm, microtype, enumerate, enumitem, color, soul, nicefrac}

\usepackage[colorlinks=true, pdfstartview=FitV, linkcolor=black, citecolor=black, urlcolor=blue]{hyperref}
\usepackage{mathtools} 
\usepackage{bbm}

\newcommand{\rev}[1]{{#1}} 

\usepackage[T1]{fontenc}
\usepackage{lmodern, microtype}
\usepackage[small]{titlesec}
\usepackage[left=1.3in, right=1.3in, top=1.3in, bottom=1.3in]{geometry}

\hypersetup{
  colorlinks=true,
  linkcolor=blue,
  citecolor=blue 
}

\def\pmb#1{\noindent
    \setbox0=\hbox{$#1$}%
    \kern-.025em\copy0\kern-\wd0
    \kern .05em\copy0\kern-\wd0
    \kern-.025em\raise .0433em\box0}

\def\mathtester#1{\relax\ifmmode #1 \else $#1$\fi}

\def\Cal#1{\mathtester{\mathcal #1}}
\def\Bol#1{\mathtester{\mathbb #1}}

\def\M{{\Cal M}}        
        \def\P{{\Cal P}}

\def\EE{{\Bol E}}

        \def\NN{{\Bol N}}
        \def\PP{{\Bol P}}   
        \def\RR{{\Bol R}}

\def\dd{{\bf d}}		 
\def\ee{{\bf e}}		 

\def\argmax{\mathop{\rm argmax}}		
					\def\|{\,|\,}

\def\dd{\mathrm{d}}
\def\Var{\mathrm{Var}}

\def\ee{\mathrm{e}}

\newtheorem{theorem}{Theorem}
\newtheorem{lemma}{Lemma}

\newtheorem{corollary}{Corollary}
\newtheorem{proposition}{Proposition}

\newtheorem{remark}{Remark}

\newcommand{\ind}[1]{{\mathbbm{1}_{\left\{{#1}\right\}}}}

\newtheorem{definition}{Definition}
\let\olddefinition\definition
\renewcommand{\definition}{\olddefinition\normalfont}

\newcommand{\fosd}{{1}}
\newcommand{\sst}{{2}}

\usepackage{natbib}
\setcitestyle{authoryear,open={(},close={)}}

\title{\bf \Large Stochastic Dominance Under Independent Noise\thanks{We are grateful to the editor and the referees for their comments and suggestions. We also thank Kim Border, Simone Cerreia-Vioglio, Jak\u{s}a Cvitanic, Ed Green, Elliot Lipnowski, Massimo Marinacci, Doron Ravid and Xiaosheng Mu, as well as seminar audiences at the Workshop on Information and Social Economics at Caltech, the University of Chile, Princeton University, and the Stony Brook International Conference on Game Theory. All errors and omissions are our own.}}

\author{ \large Luciano Pomatto\thanks{Caltech.} \ \ \ %
Philipp Strack\thanks{UC Berkeley.} \ \ \ %
Omer Tamuz\thanks{Caltech. Omer Tamuz was supported by a grant from the Simons Foundation (\#419427).}}

\begin{document}

\maketitle

\begin{abstract}
Stochastic dominance is a crucial tool for the analysis of choice under risk. It is typically analyzed as a property of two gambles that are taken in isolation. We study how additional independent sources of risk (e.g. uninsurable labor risk, house price risk, etc.) can affect the ordering of gambles. We show that, perhaps surprisingly, background risk can be strong enough to render lotteries that are ranked by their expectation ranked in terms of first-order stochastic dominance. We extend our results to second order stochastic dominance, and show how they lead to a novel, and elementary, axiomatization of mean-variance preferences.

\end{abstract}

\section{Introduction}
\label{sec:intro}


A choice between risky prospects is often difficult to make. It is perhaps even harder to predict, since decision makers vary in their preferences. One exception is when the  prospects are ordered in terms of stochastic dominance. In this case, a standard prediction is that the dominant option is chosen. For this reason, stochastic dominance has long been seen as a central concept in decision theory and the economics of information, and remains an area of active research
\citep[see, e.g.][among others]{muller2016between,cerreia2016stochastic}. 
More generally, stochastic dominance is an important tool for non-parametric comparison of distributions.

In the typical analysis of choices under risk, stochastic dominance is studied as a property of two prospects $X$ and $Y$ that are taken in isolation, without considering  other risks the decision maker might be facing. However, a choice between two gambles or assets often happens in the presence of unavoidable background risks, such as risks related to health, human capital, or exposure to the aggregate uncertainty in the economy.


In this paper we study how stochastic dominance is affected by background noise. Given two random variables $X$ and $Y$, we are interested in understanding if an independent risk $Z$ can make the resulting variables $X+Z$ and $Y+Z$ ordered in terms of stochastic dominance, and, if so, under what conditions. Our main result is that such a $Z$ exists whenever $X$ has higher mean than $Y$. We further show that if $X$ and $Y$ have equal mean, but the first has lower variance, then $Z$ can be chosen so that $X + Z$ dominates $Y + Z$ in terms of second-order stochastic dominance. The conditions of having higher mean or lower variance are necessary for the results to hold, and no other assumptions are imposed on $X$ and $Y$.

We show that the distribution of $Z$ can be obtained as a convex combination $(1 - \varepsilon)G + \varepsilon H$, where $G$ is the distribution of a Gaussian random variable with mean zero and (possibly large) variance, and $H$ is the distribution of an additional noise term with heavy tails. The weight $\varepsilon$ can be taken to be arbitrarily small. Hence, up to a vanishing degree of error, the background risk is indistinguishable from Gaussian/Normal noise. The role played by $H$ is to produce a distribution displaying sufficiently thick tails, a feature that, as we show, is necessary for the background noise to lead to stochastic dominance.


Interestingly, \cite{acemoglu2017microeconomic} provide evidence of key macroeconomic variables---presumably related to individual risk---that take a form similar to the one described above: their distributions are approximately Gaussian around their mean but display heavier tails. Our results, therefore, draw a connection between the analysis of choice under risk and the study of distributions with thick tail, a subject of long and renewed interest across different fields of economics, including macroeconomics \citep{gabaix2006institutional, morris2016crises}, finance \citep{gabaix2003theory, kelly2014tail}, as well as other disciplines \citep{nair2013fundamentals}.

As is well-known, ranking risky prospects solely in terms of their mean and variance is a crude approach for decision-making under risk, especially if compared to expected utility theory. Nevertheless, mean and variance remain key statistics driving the decisions of practitioners and investors. 
Our results show that such a seemingly ad-hoc approach can be justified  
in the presence of suitable background risk: while first and second order stochastic dominance are much stronger orders than the comparisons of means and variances, our results suggest that the latter are, under background risk, surprisingly good proxies for the former.

To the best of our knowledge, \rev{the only other paper to study the effect of independent noise on stochastic dominance is the contemporaneous, and independent, work of \cite{tarsneyexceeding}, who provides interesting examples of binary gambles that become ranked in terms of stochastic dominance after the addition of suitable background risk.}






The paper is organized as follows. Section \ref{sec:defs} presents the main results.  Section \ref{sec:mean-var} applies our results to provide a simple axiomatization of mean-variance preferences, distinct from the previous approaches \citep*[][]{maccheroni2006ambiguity}. In Section \ref{sec:auctions} we apply our results to the study of mechanism design with risk averse agents. \rev{In Section \ref{sec:uniform} in the appendix we restrict attention to bounded gambles and provide a more quantitative version of our main result.}

\section{Definitions and Main Results}
\label{sec:defs}

A random variable (or gamble) $X$ \emph{first-order stochastically dominates} a random variable $Y$, denoted $X \geq_\fosd Y$, if it satisfies
$
\EE[\phi(X)] \geq \EE[\phi(Y)]
$
for every increasing function $\phi$ for which the two expectations are well defined. We write $X >_\fosd Y$ if $X$ and $Y$ have distinct distributions (equivalently, if $X \geq_1 Y$ but $Y \not \geq_1 X$).\footnote{Yet again equivalently, $X >_\fosd Y$ if $\EE[\phi(X)] > \EE[\phi(Y)]$ for all strictly increasing functions $\phi$ for which the two expectations are well defined.} A random variable $X$ \emph{second-order} stochastically dominates $Y$, denoted $X \geq_\sst Y$, if $\EE[\phi(X)] \geq \EE[\phi(Y)]$ for every {\em concave} increasing function $\phi$ for which the two expectations are well defined.  As before, we write $X >_\sst Y$ if $X$ and $Y$ have distinct distributions.

\rev{We make a few observations which will be useful in what follows. Given two random variables $X$ and $Y$, if $X \geq_\fosd Y$ then $X + Z \geq_\fosd Y + Z$ for any $Z$ that is independent from the two.\footnote{If $X \geq_\fosd Y$ then for every increasing  $\phi$ and $z \in \RR$, $\EE[\phi(X + z)] \geq \EE[\phi(Y + z)]$ by applying the definition of first-order stochastic dominance to the function $x \mapsto \phi(x+z)$. Hence, $\EE[\phi(X+Z)] = \int \EE[\phi(X + z)] \dd F_Z(z) \geq \int \EE[\phi(Y + z)] \dd F_Z(z) = \EE[\phi(Y+Z)]$ and thus $X + Z \geq_\fosd Y + Z$.} Hence, two variables that are ranked in terms of stochastic dominance remain so after the addition of independent noise. Given a random variable $Z$, we say that $Z'$ is \textit{noisier than} $Z$ if $Z' = Z + W$ for some random variable $W$ that is independent of $Z$. Hence, if $X + Z >_{\fosd} Y+Z$ for some $Z$ that is independent of $X$ and $Y$, then $X + Z' >_{\fosd} Y+Z'$ for any independent $Z'$ that is noisier than $Z$. The analogous conclusions hold for second-order stochastic dominance. }

We denote by $\P_\infty$ the set of random variables that have all finite moments. That is, a gamble $X$ belongs to $\P_\infty$ if $\EE[|X|^n]$ is finite for all $n \in \NN$. 
We can now state our first main result:

\begin{theorem}\label{thm1}
Let $X$ and $Y$ be random variables with finite expectation. If $\EE[X] > \EE[Y]$ then there exists a random variable $Z$ that is independent from $X$ and $Y$ and such that
\begin{equation}\label{eq.thm.1}
    X + Z >_{\fosd} Y + Z.
\end{equation}
\rev{In particular, $X + Z' >_{\fosd} Y + Z'$ for every $Z'$ that is independent of $X$ and $Y$ and noisier than $Z$. Additionally, if $X,Y \in \P_\infty$ then $Z$ can be taken to belong to $\P_\infty$.}
\end{theorem}
By taking $\phi$ to be linear, the converse to the result is also true: if there exists a variable $Z$ for which (\ref{eq.thm.1}) holds, then two distributions must satisfy $\EE[X] > \EE[Y]$. When $Z$ has finite expectation, this  follows from the fact that any two variables that are ranked by $>_{\fosd}$ are also ranked by their expectation.




The next result, an immediate corollary of Theorem~\ref{thm1}, shows that large background risk can lead to risk neutral behavior: Given a finite set of gambles, there is a source of background risk such that any agent whose preferences are monotone with respect to first-order stochastic dominance will rank as more preferable gambles with higher expectation.
\begin{corollary}\label{c.finite}
Let $(X_1,\ldots,X_k)$ be gambles in $\P_\infty$ such that $\EE[X_i] > \EE[X_j]$ if $i > j$. Then there is an independent $Z \in \P_\infty$ such that $$X_i + Z >_\fosd X_j + Z \text{~~for all~~} i > j.$$
\end{corollary}
\begin{proof}
  For each $i$ there is an \rev{gamble} $Z_i \in \P_\infty$ independent of $X_{i}$ and $X_{i+1}$ such that $X_{i+1} + Z_i >_\fosd X_i + Z_i$. \rev{We can further assume that each $Z_i$ is independent of $(Z_j)_{j \neq i}$ and of $(X_1,\ldots,X_k)$.}
  Let $Z = Z_1 + \cdots + Z_k$. Then\rev{, as $Z$ is noisier than $Z_i$}, $X_{i+1} + Z_i >_\fosd X_i + Z_i$ implies $X_{i+1} + Z >_\fosd X_i + Z$. Because $>_\fosd$ is transitive, $X_i + Z >_\fosd X_j +Z$ whenever $i > j$.
\end{proof}


Corollary~\ref{c.finite} is reminiscent of the classical Arrow-Pratt approximation. As is well known, any expected utility maximizer, when choosing out of a finite menu of monetary gambles that are sufficiently small, will behave approximately as a risk neutral agent, and select as optimal the gamble with the highest expected value.\footnote{More precisely, consider a set of gambles $\{kX_1,\ldots,kX_n\}$ where $k \geq 0$ measures the size of the risk. Under expected utility and a differentiable utility function, the certainty equivalent of each $kX_i$ is given by $k\EE[X_i]$ plus a term vanishing at rate $k^2$.}
Corollary~\ref{c.finite} establishes a similar conclusion for the case where the gambles under consideration are coupled with a suitable background risk $Z$. It implies that a decision maker who has preferences that are consistent with first-order stochastic dominance---a class much larger than expected utility\footnote{Commonly used examples in this large class are cumulative prospect theory preferences \citep{tversky1992advances}, rank dependent utility \citep{quiggin1991comparative} and cautious expected utility \citep{cerreia2015cautious}.}---will behave like a \text{risk-neutral expected utility maximizer} when facing some (potentially large) background risks. 


Our second main result parallels Theorem \ref{thm1} and establishes an analogous conclusion for second-order stochastic dominance:

\begin{theorem}\label{thm.ssod}
Let $X$ and $Y$ be random variables with finite variance. If $\EE[X] = \EE[Y]$ and $\Var[X]<\Var[Y]$ then there exists a random variable $Z$ that is independent from $X$ and $Y$ and such that
\[
    X + Z >_\sst Y + Z.
\]
In particular, $X + Z' >_{\sst} Y + Z'$ for every $Z'$ that is independent of $X$ and $Y$ and noisier than $Z$. Additionally if $X,Y \in \P_\infty$ then $Z$ can be taken to belong to $\P_\infty$.
\end{theorem}

\subsection{Proof Sketch and Discussion}

Underlying Theorem \ref{thm1} is the following intuition. As is well known, first-order stochastic dominance between $X+Z$ and $Y+Z$ is equivalent to the requirement that the cumulative distribution function (or \emph{cdf}) $F_{Y+Z}$ of the gamble $Y+Z$ is greater, pointwise, than the cdf $F_{X+Z}$ of   $X+Z$.

The assumption that $X$ has higher expectation than $Y$ implies, using integration by parts, that the cdfs of the two random variables must satisfy
\begin{equation}\label{eq.int1}
    \EE[X] - \EE[Y] = \int_{-\infty}^{\infty}(F_Y(t) - F_X(t)) \dd t > 0.
\end{equation}
So, on average, the cumulative distribution function $F_Y$ must lie above $F_X$. Now consider adding an independent random variable $Z$, distributed according to a probability density $f_Z$. Then, given a point $s \in \RR$, the difference between the resulting cdfs can be expressed as
\begin{equation}\label{eq.int2}
    F_{Y+Z}(s) - F_{X+Z}(s) = \int_{-\infty}^\infty (F_Y(t) - F_X(t))f_Z(s-t) \dd t
\end{equation}
If $f_Z$ is sufficiently diffuse (for instance, by taking $Z$ to be uniformly distributed around $s$ over a sufficiently large support)
then it follows from the strict inequality (\ref{eq.int1}) that the difference (\ref{eq.int2}) is positive in a neighborhood of $s$.

So, the crucial difficulty in establishing Theorem \ref{thm1} is to show the existence of a well-defined distribution such that the difference $F_{Y+Z} - F_{X+Z}$ is positive everywhere. The existence of such a distribution $F_Z$ is not trivial. The proof of Theorem \ref{thm1} provides an explicit construction, building on mathematical techniques introduced, in a different context, by \cite{ruzsa1988algebraic}.

While the details of the construction are somewhat technical, the background risk in Theorem \ref{thm1} can be approximated (in terms of total variation distance) by Gaussian noise, up to a vanishing degree of error:

\begin{remark}\label{thm:gaussian}
Let $X$ and $Y$ be as in Theorem \ref{thm1}. For every $\varepsilon \in (0,1)$ \rev{there is a $\sigma^2>0$  and a distribution with cdf $H$ such that if $G$ is the cdf of a Gaussian distribution with zero mean and variance $\sigma^2$, and if $Z$ is independent of $X$ and $Y$ and distributed according to
\[
     F_Z = (1 - \varepsilon)G + \varepsilon H,
\]
then $Z$ satisfies \eqref{eq.thm.1}}.
\end{remark}

In the proof of Theorem \ref{thm1}, the background risk $Z$ is defined as follows. We construct a sequence $U_1,U_2,\ldots$ of i.i.d.\ random variables and an independent geometric random variable $N$ such that $\PP[N=n]=(1-\varepsilon)\varepsilon^n$ for $n \in \{0,1,2,\ldots\}$. Letting $W$ be an independent Gaussian random variable, we define
\begin{equation}\label{eq.intuition}
    Z = W + (U_1 + U_2 + \cdots + U_N).
\end{equation}

So, the random variable $Z$ is obtained as a sum of mean-zero Gaussian noise and geometric sum of independent noise terms. With probability $1 - \varepsilon$ the variable $N$ takes value $0$ and therefore $Z$ reduces to standard Gaussian noise. With probability $(1-\varepsilon)\varepsilon^n$, $n$ additional terms $U_1 + \ldots + U_n$ contribute to the background risk. 

\subsection{Orders of Magnitude}\label{sec:orders}

An immediate question, at this point, is to understand the orders of magnitude involved. In particular, how large does the background risk need to be?
We consider, as a representative example, a decision maker who is confronted with a small gamble $X$ with positive expectation, and the choice of whether to take $X$ or to receive $0$ for sure. We show that a background risk $Z$ that is not implausibly large in many contexts (e.g. risk of unemployment or risk stemming from variation in house prices) suffices to make $X$ dominant. 

In particular, suppose $X$ pays $\$12$ and $-\$10$ with equal probabilities. Then, there exists a background risk $Z$, of the form described in \eqref{eq.intuition}, that satisfies $Z + X >_1 Z$, and has the following properties: its distribution is a mixture that gives weight $1-\varepsilon = 0.99$  to a Gaussian with standard deviation \$3,500. The standard deviation of $Z$ itself is \$3,525. If $X$ pays $\$100$ and -$\$50$ with equal probability, then we can set \rev{$1 - \varepsilon = 0.978$} and the standard deviation of $Z$ is \$4,551.%
\footnote{Further examples can be found in Table \ref{table:computation} in the Appendix. Because the proof is Theorem \ref{thm1} is constructive, these quantities can be easily computed numerically.}

\section{Mean-Variance Preferences}
\label{sec:mean-var}
In this section we apply Theorem \ref{thm.ssod} to provide a simple axiomatization of the classic mean-variance preferences of \cite{markowitz1952portfolio} and  \cite{tobin1958liquidity}.\footnote{A previous axiomatization appears in \cite{epstein1985decreasing}, for DARA preferences. Related results are due to \cite{maccheroni2006ambiguity, maccheroni2009portfolio}.}
We consider a decision maker whose preferences over monetary lotteries are described by a certainty equivalent functional $C \colon \P_\infty \to \RR$ that associates to each lottery $X$ the sure amount of money $C(X)$ that makes her indifferent between $X$ and $C(X)$.  

With slight abuse of notation, we denote by $x$ the constant random variable that takes value $x \in \RR$. For the next result, a gamble $X$ is said to be a \textit{mean preserving spread} of $Y$ if the two have the same expectation and $Y$ second-order stochastically dominates $X$.\footnote{Equivalently, $X$ is smaller than $Y$ in the convex order.}

\begin{proposition}
  \label{p.mean-variance}
  A functional $C \colon \P_\infty \to \RR$ satisfies:
  \begin{enumerate}
      \item {\emph{(Certainty)}} $C(x) = x$ for all $x \in \RR$;
      \item {\emph{(Monotonicity)}} If $X$ is a mean preserving spread of $Y$ then $C(X) \leq C(Y)$; and
      \item {\emph{(Additivity)}} If $X,Y$ are independent then $C(X + Y) = C(X) + C(Y)$;
  \end{enumerate}
 if and only if there exists $k \geq 0$ such that
 \[
 C(X) = \EE[X] - k \Var[X].
 \]
 \end{proposition}

Proposition~\ref{p.mean-variance} characterizes mean-variance preferences by three simple properties. The certainty axiom is necessary for $C(X)$ to be interpreted as a certainty equivalent. Monotonicity requires $C$ to rank as more desirable gambles that are less dispersed around their mean. The additivity axiom says that the certainty equivalent is additive for independent gambles.

A key step in the proof of Proposition \ref{p.mean-variance} is to show that when restricted to mean $0$ random variables, a functional $C$ that satisfies properties (1)-(3) is a decreasing function of the variance. This is an immediate implication of Theorem \ref{thm.ssod}.

\begin{proof}[Proof of Proposition \ref{p.mean-variance}]
  
  It is immediate to verify that properties (1)-(3) are satisfied by the representation. We now prove the converse implication.
  
  Suppose $\EE[X]=\EE[Y]$ and $\Var[X] < \Var[Y]$. Then, by Theorem \ref{thm.ssod}, there exists an independent random variable $Z$ such that $X + Z >_\sst Y + Z$. By monotonicity and additivity, this implies $C(X)+C(Z) = C(X+Z) \geq C(Y+Z) = C(Y)+C(Z)$. Hence, $C(X) \geq C(Y)$.
  
  Suppose $\EE[X] = \EE[Y]$ and $\Var[X] = \Var[Y] = \sigma^2$. We show that $C(X) = C(Y)$. Assume, as a way of contradiction, that $C(X) - C(Y) = \varepsilon > 0$. Let $(X_1,X_2,\ldots)$ and $(Y_1,Y_2,\ldots)$ be i.i.d.\ sequences where each $X_i$ is distributed as $X$, and each $Y_i$ is distributed as $Y$. Fix a random variable $Z$ independent from the two sequences, and such that $\EE[Z] = 0$ and $\Var[Z] > 0$. Because $X_i \geq_\sst X$ and $X \geq_\sst X_i$, then $C(X_i) = C(X)$. Similarly, $C(Y_i) = C(Y)$. By additivity,
  \[
  C(X_1+\cdots+X_n + Z) - C(Y_1+\cdots + Y_n) = n \varepsilon + C(Z),
  \]
  which is strictly positive for $n$ large enough. However, $n \sigma^2 + \Var[Z] = \Var[X_1+\cdots+X_n + Z] > \Var[Y_1+\cdots + Y_n] = n \sigma^2$, and so, by the first part of the proof, we have that $C(X_1+\ldots+X_n + Z) - C(Y_1+\ldots+Y_n)\leq 0$ and we reached a contradiction. Hence $C(X) = C(Y)$.
  
  Thus, restricted to zero mean $X \in \P_\infty$, $C$ satisfies $C(X) = f(\Var[X])$ for some function $f$. In addition, by the first part of the proof, $f$ is nonincreasing. Furthermore, $f$ is additive (i.e., $f(x+y) = f(x)+f(y)$) since $C$ is additive. As is well known, every nonincreasing additive $f \colon \RR_+ \to \RR$ is linear.\footnote{To see this, note that if $m,n \in \mathbb{N}$, then, letting $q = m/n$, we have $mf(1) = f(m) = f(nq) = nf(q)$. Thus, $f(q) = q f(1)$ for every rational $q$. Hence, for every $x \in \RR_+$ and positive rational $q,q'$ such that $q \leq x \leq q'$ we have $q'f(1) = f(q') \leq f(x) \leq f(q) = qf(1)$. So, $f(x)=xf(1)$ for every $x$.}
  To conclude the proof, notice that for any $X$ (with possibly non-zero mean) additivity and certainty imply $C(X - \EE[X]) = C(X) - \EE[X]$. Therefore, letting $k= - f(1)$, we obtain $C(X) = \EE[X] - k\Var[X]$.
\end{proof}

We end this section with an additional result characterizing the expectation as the unique functional on $\P_\infty$ that is monotone with respect to first order stochastic dominance, and additive for independent random variable.
\begin{proposition}
  \label{p.mean}
  A functional $C \colon \P_\infty \to \RR$ satisfies:
  \begin{enumerate}
      \item {\emph{(Certainty)}} $C(1) = 1$;
      \item {\emph{(Monotonicity)}} If $X >_\fosd Y$ then $C(X) \geq C(Y)$; and
      \item {\emph{(Additivity)}} If $X,Y$ are independent then $C(X + Y) = C(X) + C(Y)$;
  \end{enumerate}
 if and only if $C(X) = \EE[X]$.
\end{proposition}

\begin{proof}[Proof of Proposition \ref{p.mean}]
  This proof closely follows that of Proposition \ref{p.mean-variance}. As in that proof, it is immediate to verify that properties (1)-(2) are satisfied by the representation. 
  
  Denote (as above) by $x$ the random variable that take the value $x \in \RR$ with probability 1, and define $f \colon \RR \to \RR$ by $f(x) = C(x)$. By additivity and monotonicity $f$ is monotone increasing and additive and so, as in the proof of Proposition~\ref{p.mean-variance}, $f(x) = k x$ for some $k \geq 0$. Certainty implies $k = 1$.
  
  We show that if $\EE[X] > \EE[Y]$ then $C(X) \geq C(Y)$. Indeed, in this case, by Theorem~\ref{thm1} there is an independent $Z \in \P_\infty$ such that $X+Z >_\fosd Y+Z$, and so by monotonicity $C(X+Z) \geq C(Y+Z)$. By additivity it follows that $C(X) \geq C(Y)$. Hence for any $\varepsilon>0$ it holds that 
  $$C(X) \leq C(\EE[X]+\varepsilon) = f(\EE[X]+\varepsilon)$$ 
  and 
  $$C(X) \geq C(\EE[X]-\varepsilon) = f(\EE[X]-\varepsilon).$$ 
  Thus $C(X) = f(\EE[X]) = \EE[X]$.
\end{proof}

\section{Implementation with Risk Aversion}
\label{sec:auctions}
In this section we show how Theorem~\ref{thm1} can be used to construct mechanisms that are robust to uncertainty about the agents' risk attitudes. 
%

Consider a mechanism design problem with $n$ agents, where each agent's type is $\theta_i \in \Theta_i$, and $\theta=(\theta_1,\ldots,\theta_n)$ is drawn from some joint distribution. We assume that the set of types $\Theta = \prod_i \Theta_i$ is finite. The designer chooses an allocation $x \in X$ and a transfer $t_i \in \RR$ for each agent $i$. By the revelation principle we can without loss of generality restrict to mechanisms where each agent reports their type and it is optimal for the agents to be truthful.
A direct mechanism $(x,t)$ is a tuple consisting of an allocation function $x: \prod_{i=1}^n \Theta_i \to X$ and a transfer function $t: \prod_{i=1}^n \Theta_i \to \Delta(\RR^n$). We restrict attention to mechanisms where each random transfer $t_i(\theta)$ has all moments. Most of the literature on mechanism design focuses on the quasi-linear case, where agent $i$'s utility is given by
\[
    v_i(x, \theta) - t_i \,.
\]
Here, $v_i(x,\theta)$ denotes the monetary equivalent of the utility agent $i$ derives from the physical allocation $x$ when the type profile equals $\theta=(\theta_1,\ldots,\theta_n$). Note, that as $v$ depends on the complete type profile we allow for interdependent values.

A restriction which is imposed when assuming quasi-linear preferences is that the agents are risk neutral over money. To more generally model agents who might not be risk neutral, we assume that preferences are of the form
\begin{equation}\label{eq:utility}
        u_i \left( v_i(x, \theta) - t_i \right)\,,
\end{equation}
where $u_i : \RR \to \RR$ is agent $i$'s utility function over money. We assume $u_i$ is bounded by a polynomial; see Section~\ref{sec:Discussion} for a discussion of this integrability assumption. Preferences \eqref{eq:utility} are commonly considered in the literature on auctions with risk aversion \citep[see for example case 1.\ in][]{maskinriley1984} or in the literature on optimal income taxation \citep[see for example][]{diamond1998}. 

A mechanism $(x,t)$ is \emph{Bayes incentive compatible} if for every type $\theta_i$ and every agent $i$ it is optimal to report her type truthfully to the mechanism, given that all other players do the same
\begin{equation}\label{eq:u-bic}
    \theta_i \in \argmax_{\hat{\theta}_i} \EE \Big[u_i \left(\,v_i(x(\hat{\theta}_i,\theta_{-i}), \theta) - t_i (\hat{\theta}_i, \theta_{-i}) \,\right) \mid \theta_i\Big] \tag{$u$-BIC}
\end{equation}
A mechanism is Bayes incentive compatible (BIC) \textit{under quasi-linear preferences} if it satisfies \eqref{eq:u-bic} when all agents' utilities over money are linear, i.e. $u_i(x)=x$. 
A mechanism is \emph{strictly} Bayes incentive compatible if the maximum in \eqref{eq:u-bic} is unique.

Bayes incentive compatibility depends on the utility functions $u=(u_1,\ldots,u_n)$ and thus the risk attitudes of the agents. 
We are interested in finding mechanisms 
which implement a given physical allocation $x$ for any vector of 
utilities. 
%
%
 To explicitly model this problem we consider the stronger notion of \emph{ordinal incentive compatibility} \citep{d1988ordinal}.

\begin{definition} A direct mechanism $(x,t)$ is ordinal incentive compatible if \eqref{eq:u-bic} is satisfied for all non-decreasing utility functions $u_1,\ldots,u_n$. 
\end{definition}
It is natural to ask which physical allocation rules $x$ can be implemented ordinally.\footnote{A first observation is that every mechanism which is ex-post or dominant strategy IC is also ordinally IC.} Our main result in this section is that (essentially) any allocation that can be implemented when the agents have quasi-linear utilities 
can also be implemented when the agents have arbitrary utilities $u$ that are unknown to the designer. Thus, perhaps surprisingly, knowledge of the agents risk attitudes is not necessary to implement a given allocation. To make a Bayes IC mechanism ordinally IC it suffices to add a carefully chosen risk to the transfer that is independent of the agents' reports. 

\begin{proposition}\label{prop:robust-implementation}
 Suppose the direct mechanism $(x,t)$ is strictly Bayes incentive compatible with quasi-linear preferences. Then, there exists another \textbf{ordinally incentive compatible} direct mechanism $(x,\tau)$ that implements the same physical allocation $x$ and raises the same expected revenue from each agent given each vector of reported types $\hat{\theta}$, i.e., $
 	\EE[\tau_i(\hat{\theta})] = \EE[t_i(\hat{\theta})]$ for every $i$.
\end{proposition}

\begin{proof}
The proof of Proposition \ref{prop:robust-implementation} is a direct application of Corollary \ref{c.finite}. Fix agent $i$'s type to be $\theta_i$. For each $\hat{\theta}_i \in \Theta_i$ denote by $\mu_{\hat{\theta}_i}$ the distribution of $v_i(x(\hat{\theta}_i,\theta_{-i}),\theta)-t_i(\hat{\theta}_i,\theta_{-i})$ (i.e., the monetary utility of $i$ when reporting $\hat{\theta}_i$) conditioned on the true type $\theta_i$. Let $X_{\hat{\theta}_i}$ be distributed according to $\mu_{\hat{\theta}_i}$. Strict Bayes incentive compatibility says that
$$
  \EE[X_{\theta_i}] > \EE[X_{\hat{\theta}_i}]
$$
whenever $\hat{\theta}_i \neq \theta_i$. By Corollary \ref{c.finite} there thus exists an independent random variable $Z_{\theta_i}$ such that $X_{\theta_i}+Z_{\theta_i} >_1 X_{\hat{\theta}_i}+Z_{\theta_i}$. Furthermore, since the random variables $(X_{\hat{\theta}_i})$ have all moments, we can assume that $Z_{\theta_i'}$ does too, and has zero mean. By repeating this construction for each player $i$ and each true type $\theta_i$, we obtain a collection $(Z_{\theta_i})$ of random variables which we can assume to be independent. Let $Z_i = \sum_{\theta_i \in \Theta_i}Z_{\theta_i}$. Then it still holds that $X_{\theta_i}+Z_i >_1 X_{\hat{\theta}_i}+Z_i$
for every every $\theta_i$, and so the mechanism $(x,\tau)$ with $\tau_i(\hat{\theta}) = t_i(\hat{\theta}) + Z_i$ is ordinal incentive compatible. In addition, $\EE[\tau_i(\hat{\theta})] = \EE[t_i(\hat{\theta})]$ since $Z_i$ has zero mean.
%
\end{proof}


The result has the following additional implication. As $(x,t)$ and $(x,\tau)$ implement the same physical allocation $x$, it follows that the set of strictly implementable physical allocations is the same under Bayes IC and ordinal IC. This is in contrast to dominant strategy IC for which the set of implementable physical allocations is strictly smaller.\footnote{%
Note, that as \cite{gershkov2013equivalence} show, the set of implementable interim utilities under some conditions is the same under dominant strategy and Bayesian incentive compatibility even if the set of implementable allocation rules differs.}



In our analysis we have abstracted away from individual rationality. A mechanism that is individually rational when agents are risk neutral is not necessarily individually rational when $u_i$ is concave. 
The addition of background noise may violate individual rationality for risk averse agents, and thus require the agents to be compensated to ensure their participation. This, however, 
would not change the set of implementable physical allocation rules.

\section{Discussion}\label{sec:Discussion}
\label{sec:discussion}


\paragraph{Unbounded Support and Tail Risk} We first observe that the noise term $Z$ in Theorem~\ref{thm1} cannot have bounded support, unless additional assumptions are imposed on $X$ and $Y$. This holds even if we only consider bounded $X$ and $Y$. Indeed, in order for $X + Z >_1 Y + Z$ to be satisfied, the maximum of the support of $X + Z$ must lie weakly above that of $Y + Z$. In particular, if $Z$ has bounded \rev{support} and $X + Z >_1 Y + Z$, then the maximum of the support of $X$ must be greater than the maximum of the support $Y$.

In addition, the background risk $Z$ must generally display non-negligible risk at the tails.
For instance, it is impossible for a Gaussian $Z$ to satisfy Theorem~\ref{thm2}; this holds even if we restrict the random variables in question to be finitely supported.  Indeed, any $Z$ that satisfies Theorem~\ref{thm2} must have thick tails in the sense that $\EE[\exp(tZ)] = \infty$ for some $t \in \RR$ large enough. 

To see the necessity for thick tailed distributions, note that if the maximum of the support of $X$ is strictly less than that of $Y$ (which of course does not preclude that $\EE[X] > \EE[Y]$) then  $\EE[\exp(tX)] < \EE[\exp(tY)]$ for $t$ large enough. Hence for any independent $Z$ for which $\EE[\exp(tZ)]$ is finite it holds that
$$\EE\left[\ee^{t(X+Z)}\right] = \EE\left[\ee^{tX}\right] \cdot \EE\left[\ee^{tZ}\right] < \EE\left[\ee^{tY}\right] \cdot \EE\left[\ee^{tZ}\right]  = \EE\left[\ee^{t(Y+Z)}\right],$$
and so it cannot be that $X+Z >_\fosd Y+Z$.

\paragraph{Integrability} As is well known, distributions with unbounded support require the specification of a class of utility functions for which all expectations are finite, so that issues similar to the St.\ Petersburg paradox do not arise. 

A standard solution is to restrict attention to utility functions that are bounded. See, e.g., \citet{aumann1977st}. More generally, following \citet{russell1978admissible} \citep[see also][]{ryan1974use, arrow1974use}, one may wish to consider utility functions that have polynomial tails. \rev{That is, utility functions $u : \mathbb{R} \to \RR$ that are bounded, in absolute value, by a polynomial.}
This assumption is more general and allows for utility functions that are strictly increasing and strictly concave everywhere.\footnote{Given any utility function $u : \RR \to \RR$ and an arbitrarily large bounded interval, there exists a utility $v$ that agrees with $u$ on that interval and can be taken to be bounded or with polynomial tails. So, both assumptions amount to non-falsifiable integrability conditions. %
}
For any such $u$ and gamble $X \in \P_\infty$ with all finite moments, the resulting expected utility $\EE[u(X)]$ is well-defined and finite.

By definition, distributions with thick tails rule out CARA as a class of utility functions for which integrability is preserved. This, however, does not make thick tailed distributions pathological. Many standard distributions have thick tails: examples include geometric, exponential and gamma distributions. All of these have finite moments and exponentially vanishing tails.  Distribution with thick tails have been used to describe many economic variables of interest \citep{acemoglu2017microeconomic}, and are the subject of a growing literature in economics \citep[see, e.g., ][]{morris2016crises}.

\newcommand{\possessivecite}[1]{\citeauthor{#1}'s \citeyearpar{#1}}

\rev{\paragraph{Calibration} \possessivecite{rabin2000risk} calibration theorem famously highlighted a tension between expected utility and risk aversion over small gambles. Subsequent literature clarified that the conclusions of Rabin's impossibility theorem extend to a broader class of preferences. In particular, \cite{barberis2006individual} showed how independent background risk can eliminate first-order risk aversion, a feature that in the absence of background risk can, under some conditions, accommodate risk-aversion over small gambles. Our paper contributes to this literature by showing how background risk can lead to risk-neutral behavior, for any monotone preference, and for a suitable choice of background risk.
}

\paragraph{Conclusion} Our main theorem establishes a connection between two orderings---having strictly greater mean and first-order stochastic dominance---that differ substantially in their strength and implications. An interesting avenue for future research is to study more quantitative versions of our result and to what extent the stark predictions implied by our results are reflected in actual behavior.


\bibliography{main}

\begin{thebibliography}{34}
\providecommand{\natexlab}[1]{#1}
\providecommand{\url}[1]{\texttt{#1}}
\expandafter\ifx\csname urlstyle\endcsname\relax
  \providecommand{\doi}[1]{doi: #1}\else
  \providecommand{\doi}{doi: \begingroup \urlstyle{rm}\Url}\fi

\bibitem[Acemoglu et~al.(2017)Acemoglu, Ozdaglar, and
  Tahbaz-Salehi]{acemoglu2017microeconomic}
D.~Acemoglu, A.~Ozdaglar, and A.~Tahbaz-Salehi.
\newblock Microeconomic origins of macroeconomic tail risks.
\newblock \emph{American Economic Review}, 107\penalty0 (1):\penalty0 54--108,
  2017.

\bibitem[Arrow(1974)]{arrow1974use}
K.~J. Arrow.
\newblock The use of unbounded utility functions in expected-utility
  maximization: Response.
\newblock \emph{The Quarterly Journal of Economics}, 88\penalty0 (1):\penalty0
  136--138, 1974.

\bibitem[Aumann(1977)]{aumann1977st}
R.~J. Aumann.
\newblock The st. petersburg paradox: A discussion of some recent comments.
\newblock \emph{Journal of Economic Theory}, 14\penalty0 (2):\penalty0
  443--445, 1977.

\bibitem[Barberis et~al.(2006)Barberis, Huang, and
  Thaler]{barberis2006individual}
N.~Barberis, M.~Huang, and R.~H. Thaler.
\newblock Individual preferences, monetary gambles, and stock market
  participation: A case for narrow framing.
\newblock \emph{American economic review}, 96\penalty0 (4):\penalty0
  1069--1090, 2006.

\bibitem[Cerreia-Vioglio et~al.(2015)Cerreia-Vioglio, Dillenberger, and
  Ortoleva]{cerreia2015cautious}
S.~Cerreia-Vioglio, D.~Dillenberger, and P.~Ortoleva.
\newblock Cautious expected utility and the certainty effect.
\newblock \emph{Econometrica}, 83\penalty0 (2):\penalty0 693--728, 2015.

\bibitem[Cerreia-Vioglio et~al.(2016)Cerreia-Vioglio, Maccheroni, and
  Marinacci]{cerreia2016stochastic}
S.~Cerreia-Vioglio, F.~Maccheroni, and M.~Marinacci.
\newblock Stochastic dominance analysis without the independence axiom.
\newblock \emph{Management Science}, 63\penalty0 (4):\penalty0 1097--1109,
  2016.

\bibitem[d'Aspremont and Peleg(1988)]{d1988ordinal}
C.~d'Aspremont and B.~Peleg.
\newblock Ordinal bayesian incentive compatible representations of committees.
\newblock \emph{Social Choice and Welfare}, 5\penalty0 (4):\penalty0 261--279,
  1988.

\bibitem[Diamond(1998)]{diamond1998}
P.~A. Diamond.
\newblock Optimal income taxation: An example with a u-shaped pattern of
  optimal marginal tax rates.
\newblock \emph{The American Economic Review}, 88\penalty0 (1):\penalty0
  83--95, 1998.

\bibitem[Epstein(1985)]{epstein1985decreasing}
L.~G. Epstein.
\newblock Decreasing risk aversion and mean-variance analysis.
\newblock \emph{Econometrica: Journal of the Econometric Society}, pages
  945--961, 1985.

\bibitem[Fremlin(2002)]{fremlin}
D.~H. Fremlin.
\newblock \emph{Measure theory}, volume~2.
\newblock Torres Fremlin, 2002.

\bibitem[Gabaix et~al.(2003)Gabaix, Gopikrishnan, Plerou, and
  Stanley]{gabaix2003theory}
X.~Gabaix, P.~Gopikrishnan, V.~Plerou, and H.~E. Stanley.
\newblock A theory of power-law distributions in financial market fluctuations.
\newblock \emph{Nature}, 423\penalty0 (6937):\penalty0 267, 2003.

\bibitem[Gabaix et~al.(2006)Gabaix, Gopikrishnan, Plerou, and
  Stanley]{gabaix2006institutional}
X.~Gabaix, P.~Gopikrishnan, V.~Plerou, and H.~E. Stanley.
\newblock Institutional investors and stock market volatility.
\newblock \emph{The Quarterly Journal of Economics}, 121\penalty0 (2):\penalty0
  461--504, 2006.

\bibitem[Gershkov et~al.(2013)Gershkov, Goeree, Kushnir, Moldovanu, and
  Shi]{gershkov2013equivalence}
A.~Gershkov, J.~K. Goeree, A.~Kushnir, B.~Moldovanu, and X.~Shi.
\newblock On the equivalence of bayesian and dominant strategy implementation.
\newblock \emph{Econometrica}, 81\penalty0 (1):\penalty0 197--220, 2013.

\bibitem[Kelly and Jiang(2014)]{kelly2014tail}
B.~Kelly and H.~Jiang.
\newblock Tail risk and asset prices.
\newblock \emph{The Review of Financial Studies}, 27\penalty0 (10):\penalty0
  2841--2871, 2014.

\bibitem[Maccheroni et~al.(2006)Maccheroni, Marinacci, and
  Rustichini]{maccheroni2006ambiguity}
F.~Maccheroni, M.~Marinacci, and A.~Rustichini.
\newblock Ambiguity aversion, robustness, and the variational representation of
  preferences.
\newblock \emph{Econometrica}, 74\penalty0 (6):\penalty0 1447--1498, 2006.

\bibitem[Maccheroni et~al.(2009)Maccheroni, Marinacci, Rustichini, and
  Taboga]{maccheroni2009portfolio}
F.~Maccheroni, M.~Marinacci, A.~Rustichini, and M.~Taboga.
\newblock Portfolio selection with monotone mean-variance preferences.
\newblock \emph{Mathematical Finance: An International Journal of Mathematics,
  Statistics and Financial Economics}, 19\penalty0 (3):\penalty0 487--521,
  2009.

\bibitem[Markowitz(1952)]{markowitz1952portfolio}
H.~Markowitz.
\newblock Portfolio selection.
\newblock \emph{The Journal of Finance}, 7\penalty0 (1):\penalty0 77--91, 1952.

\bibitem[Maskin and Riley(1984)]{maskinriley1984}
E.~Maskin and J.~Riley.
\newblock Optimal auctions with risk averse buyers.
\newblock \emph{Econometrica}, 52\penalty0 (6):\penalty0 1473--1518, 1984.

\bibitem[Mattner(1999)]{mattner1999what}
L.~Mattner.
\newblock What are cumulants?
\newblock \emph{Documenta Mathematica}, 4:\penalty0 601--622, 1999.

\bibitem[Mattner(2004)]{mattner2004cumulants}
L.~Mattner.
\newblock Cumulants are universal homomorphisms into {H}ausdorff groups.
\newblock \emph{Probability theory and related fields}, 130\penalty0
  (2):\penalty0 151--166, 2004.

\bibitem[Morris and Yildiz(2016)]{morris2016crises}
S.~Morris and M.~Yildiz.
\newblock Crises: Equilibrium shifts and large shocks, 2016.
\newblock Working paper.

\bibitem[M{\"u}ller et~al.(2016)M{\"u}ller, Scarsini, Tsetlin, and
  Winkler]{muller2016between}
A.~M{\"u}ller, M.~Scarsini, I.~Tsetlin, and R.~L. Winkler.
\newblock Between first-and second-order stochastic dominance.
\newblock \emph{Management Science}, 63\penalty0 (9):\penalty0 2933--2947,
  2016.

\bibitem[Nair et~al.(2013)Nair, Wierman, and Zwart]{nair2013fundamentals}
J.~Nair, A.~Wierman, and B.~Zwart.
\newblock The fundamentals of heavy-tails: properties, emergence, and
  identification.
\newblock In \emph{ACM SIGMETRICS Performance Evaluation Review}, volume~41,
  pages 387--388. ACM, 2013.

\bibitem[Quiggin(1991)]{quiggin1991comparative}
J.~Quiggin.
\newblock Comparative statics for rank-dependent expected utility theory.
\newblock \emph{Journal of Risk and Uncertainty}, 4\penalty0 (4):\penalty0
  339--350, 1991.

\bibitem[Rabin(2000)]{rabin2000risk}
M.~Rabin.
\newblock Risk aversion and expected-utility theory: A calibration theorem.
\newblock \emph{Econometrica}, 68\penalty0 (5):\penalty0 1281--1292, 2000.

\bibitem[Rachev and Stoyanov(2008)]{rachev2008advanced}
S.~Rachev and S.~Stoyanov.
\newblock Advanced stochastic models, risk assessment, and portfolio
  optimization, the frank j. fabozzi series, 2008.

\bibitem[Rachev et~al.(2013)Rachev, Klebanov, Stoyanov, and
  Fabozzi]{rachev2013}
S.~T. Rachev, L.~Klebanov, S.~V. Stoyanov, and F.~Fabozzi.
\newblock \emph{The methods of distances in the theory of probability and
  statistics}.
\newblock Springer, 2013.

\bibitem[Russell and Seo(1978)]{russell1978admissible}
W.~R. Russell and T.~K. Seo.
\newblock Admissible sets of utility functions in expected utility
  maximization.
\newblock \emph{Econometrica: Journal of the Econometric Society}, pages
  181--184, 1978.

\bibitem[Ruzsa and Sz{\'e}kely(1988)]{ruzsa1988algebraic}
I.~Ruzsa and G.~J. Sz{\'e}kely.
\newblock \emph{Algebraic probability theory}.
\newblock John Wiley \& Sons Inc, 1988.

\bibitem[Ryan(1974)]{ryan1974use}
T.~M. Ryan.
\newblock The use of unbounded utility functions in expected-utility
  maximization: comment.
\newblock \emph{The Quarterly Journal of Economics}, 88\penalty0 (1):\penalty0
  133--135, 1974.

\bibitem[Shaked and Shanthikumar(2007)]{shaked2007stochastic}
M.~Shaked and J.~G. Shanthikumar.
\newblock \emph{Stochastic orders}.
\newblock Springer Science \& Business Media, 2007.

\bibitem[Tarsney(2019)]{tarsneyexceeding}
C.~Tarsney.
\newblock Exceeding expectations: Stochastic dominance as a general decision
  theory.
\newblock 2019.

\bibitem[Tobin(1958)]{tobin1958liquidity}
J.~Tobin.
\newblock Liquidity preference as behavior towards risk.
\newblock \emph{The Review of Economic Studies}, 25\penalty0 (2):\penalty0
  65--86, 1958.

\bibitem[Tversky and Kahneman(1992)]{tversky1992advances}
A.~Tversky and D.~Kahneman.
\newblock Advances in prospect theory: Cumulative representation of
  uncertainty.
\newblock \emph{Journal of Risk and uncertainty}, 5\penalty0 (4):\penalty0
  297--323, 1992.

\end{thebibliography}

\newpage 
\appendix
\section{Proof of Theorem \ref{thm1}}
\label{sec:proof-1}
Denote by $\P_n$ the collection of all Borel probability measures on $\RR$ that have finite $n$th moment:
\[
   \P_n = \left\{\nu \,:\, \int_{\RR}|x|^n\,\dd \nu(x) < \infty \right\}.
\]
We denote $\P_\infty = \bigcap_n \P_n$. Likewise, denote by $\M_n$ the collection of all bounded Borel signed measures on $\RR$ that have finite $n$th moment. That is, $\mu \in \M_n$ if $\int_{\RR}|x|^n\,\dd \mu(x) < \infty$. Recall that a signed measure $\mu$ is bounded if its absolute value $|\mu|$ is a finite measure. As usual, given $\mu,\nu \in \M_n$ we write $\mu \geq \nu$ if $\mu(A) \geq \nu(A)$ for every Borel $A \subseteq \RR$. We equip $\M_n$ with the total-variation norm 
\[
 \Vert \mu \Vert = \int_{\RR} 1\,\dd|\mu|.
\]
We denote the convolution of $\mu_1,\mu_2 \in \M_n$ (and in its subset $\P_n$) by $\mu_1*\mu_2$, and by $\mu^{(k)}$ the $k$-fold convolution of $\mu$ with itself for all $k\geq 1$. To simplify notation we define $\mu^{(0)}$ to be $\delta$, the Dirac measure at zero, so that $\mu^{(k)}*\mu^{(m)} = \mu^{(k+m)}$ for all $k,m \geq 0$. Note that $\M_n$, equipped with the norm defined above, is a Banach space which is closed under convolutions. In fact, $(\M_n,*,+)$ is a Banach algebra, so that $\mu * (\nu_1+\nu_2) = \mu * \nu_1 + \mu * \nu_2$.

The following lemma is due to \citet[pp. 126--127]{ruzsa1988algebraic}. It states that a signed measure that assigns total mass $1$ to $\RR$ can be ``smoothed'' into a probability measure by convolving it with an appropriately chosen probability measure. We provide the proof for the reader's convenience; an essentially identical proof also appears in \citet[p.\ 616]{mattner1999what}, as well as in \citet[p.\ 159]{mattner2004cumulants}. 
\begin{lemma}[Ruzsa and Sz{\'e}kely, Mattner]\label{l.mattner}
Let $n \in \{0,1,\ldots,\infty\}$. For every $\mu \in \M_n$ with $\mu(\RR)=1$ there is a $\nu \in \P_n$ such that $\mu * \nu \in \P_n$.
\end{lemma}
\begin{proof}
  Let $\rho_a$ be the measure of a Gaussian random variable with mean zero and standard deviation $a > 0$.\footnote{One could take here any other distribution (e.g., the uniform probability distribution on $[-a,a]$) such that $\rho_a^{(2)} \geq \beta\rho_a$ for some $\beta>0$.}.  For some $0 < c < 1$ and some $\pi \in \P_n$ let
  \[
    \tau = (1-c)\sum_{k=0}^\infty c^k \pi^{(k)},
  \]
  where $\pi^{(0)}$ is $\delta$, the Dirac measure at zero. Since $c<1$ the series converges and so $\tau$ is a probability measure.
  Let
  \[
    \nu = \rho_a^{(2)} * \tau.
  \]
  We show that $\mu * \nu \in \P_n$ for an appropriate choice of $a,c$ and $\pi$. To see that $\mu * \nu$ is a probability measure note first that $[\mu * \nu](\RR)=1$,\footnote{This follows immediately from the fact that $\int_{-\infty}^\infty \dd (\mu * \nu)(x) = \int_{-\infty}^\infty \dd \mu (x) \times \int_{-\infty}^\infty \dd \nu (x)\,.$} and so it suffices to show that $\mu * \nu$ is positive. To see this, we write
  \begin{eqnarray*}
    \mu * \nu &=& \rho_a^{(2)} * \mu * \tau\\
    &=& \rho_a^{(2)} * (\mu -\delta + \delta - c \pi + c \pi) * \tau\\
    &=& \rho_a^{(2)} * (\mu -\delta + c \pi)*\tau +\rho_a^{(2)} * (\delta - c \pi )*\tau.
  \end{eqnarray*}
  Now, by the definition of $\tau$, 
  \[
    c \pi * \tau = (1-c)\sum_{k=0}^\infty c^{k+1}\pi^{(k+1)} = \tau-(1-c)\delta.
  \]
  Hence $(\delta - c \pi )*\tau = (1-c)\delta$, and so
  \[
    \mu * \nu = \rho_a^{(2)} * (\mu -\delta + c \pi)*\tau + (1-c)\rho_a^{(2)}.
  \]
  It follows by comparing the densities of $\rho_a^{(2)}$ and $\rho_a$ that $\rho_a^{(2)} \geq \frac{1}{\sqrt{2}}\rho_a$. Hence
  \begin{eqnarray}
    \rho_a^{(2)} * (\mu -\delta + c \pi) 
    &\geq& \rho_a^{(2)} * (\mu - \delta) + \frac{c}{\sqrt 2} \rho_a * \pi\\
    &=& \frac{c}{\sqrt 2}  \rho_a * \left[\frac{\sqrt 2}{c} \rho_a * (\mu -\delta) +  \pi\right]. 
  \end{eqnarray}
  Thus, if we choose $a, c$ and $\pi$ so that $\frac{\sqrt 2}{c} \rho_a * (\mu -\delta) +  \pi$ is a positive measure it will follow that $\mu * \nu$ is also positive. To this end we set
  \[
    \frac{c}{\sqrt 2} = \Vert \rho_a*(\mu-\delta) \Vert.
  \]
  and
  \[
    \pi = \frac{\sqrt 2}{c}|\rho_a * (\mu - \delta)| = \frac{1}{\Vert \rho_a * (\mu - \delta) \Vert }|\rho_a * (\mu - \delta)|,
  \]
  and if $0 < c < 1$. If $c=0$ for some $a$ then $\rho_a*\mu=\rho_a$ and we can take $\nu=\rho_a$ to conclude the proof of the theorem. In addition, $c = \sqrt 2 \Vert (\mu-\delta)*\rho_a \Vert$ tends to $0$ as $a$ tends to infinity,%
  \footnote{Let $\delta_y$ be the point mass at $y$. Then $\delta_y * \rho_a$ is a Gaussian distribution with mean $y$ and standard deviation $a$. The Kullback-Leibler divergence between the two is well-known to be $D(\delta_y * \rho_a \Vert \rho_a) = \frac{1}{2} \left(\frac{y}{a}\right)^2$. By Pinsker's Inequality,
\[
    \Vert \delta_y * \rho_a - \rho_a \Vert \leq \sqrt{2 D(\delta_y * \rho_a \Vert \rho_a)} = \frac{|y|}{a}\cdot
\]
Hence $\Vert \delta_y * \rho_a - \rho_a \Vert \leq \min\{2,  |y|/a\}$. Because $\mu(\RR)=1$, $[\mu * \rho_a](x)-\rho_a(x) = \int [\delta_y * \rho_a](x)-\rho_a(x) \,\dd \mu(y)$, and so
$$
    \Vert \mu*\rho_a-\rho_a \Vert \leq \int_{\RR} \Vert \delta_y * \rho_a - \rho_a \Vert\,\dd |\mu |(y) \leq \int_{\RR} \min\{2, |y|/a\} \,\dd|\mu|(y),
$$
which tends to $0$ as $a$ tends to infinity.
  }
  so we can choose $a$ large enough so that $c < 1$, and in fact $c$ as small as we like.
  
  We have shown that $\mu * \nu$ is a sum of positive measures, hence positive, and hence a probability measure. It remains to be shown that $\nu$ has finite $n$th moment, and hence is in $\P_n$. To this end, note first that
  \[
    \pi = \frac{\sqrt 2}{c}|\rho_a * (\mu - \delta)| \leq \frac{\sqrt 2}{c}(\rho_a * (|\mu| +\delta)) \in \M_n,
  \]
  and so $\pi \in \P_n$.  Since the $n$th moment of $\pi^{(k)}$ is at most $k^n$ times the $n$th moment of $\pi$, it follows that $\tau \in \P_n$, and so $\nu = \rho_a^{(2)} * \tau \in \P_n$.
\end{proof}
It is important that the proof of this lemma is constructive; indeed, we get that 
$$
    \nu =  (1-c) \rho_a^{(2)} * \big( \delta + c \pi + c^2 \pi^{(2)} + \cdots\big),
$$
for $c$ that can be arbitrarily small, and $a$ that may need to be correspondingly large. Since $\rho_a$ is the distribution of a Gaussian with zero mean and standard deviation $a$,  $\nu$ is the distribution of 
$$
    Z = W + (X_0+X_1+ \cdots + X_N),
$$
where $W$ is a Gaussian with mean zero and standard deviation $\sqrt{2}a$, $X_0 = 0$, $X_1, X_2, \ldots$ are i.i.d.\ random variables with distribution $\pi$ (defined in the proof), and $N$ is geometric with parameter $c$. In particular, when $c$ is small, $Z$ is close in total variation to $W$, as $\Vert \nu - \rho_a^{(2)} \Vert \leq c$. This shows Remark~\ref{thm:gaussian}, taking $c = \varepsilon$.

\bigskip

We now proceed with the proof of Theorem \ref{thm1}. Let $X$ and $Y$ be two random variables with $k = \EE[X] - \EE[Y] > 0$.  Define the signed measure $\sigma$ as
\[
    \sigma(A) = \frac{1}{k}\int_A F_Y(t) - F_X(t)\,\dd t
\]
for every $A \subseteq \RR$. By Tonelli's Theorem
$$\EE[X] =  \EE[X^+]-\EE[X^-] = \int_0^{\infty} 1 - F_X(t) \dd t - \int_{-\infty}^0 F_X(t) \dd t$$ 
so
\[
    \EE[X] - \EE[Y] = - \int_{\infty}^0 (F_X(t) - F_Y(t)) \dd t + \int_0^{\infty} (F_Y(t) - F_X(t)) \dd t = \int_{\RR} F_Y(t) - F_X(t)\,\dd t.
\]
Hence $\sigma(\RR) = 1$. Furthermore, $\sigma$ is a bounded measure, since
\begin{eqnarray*}
  k|\sigma|(\RR) &=& \int_{-\infty}^\infty |F_Y(t) - F_X(t)|\,\dd t\\
  &\leq& \int_{-\infty}^0 \big[F_Y(t) + F_X(t)\big]\,\dd t + \int_{0}^\infty \big[(1-F_Y(t)) + (1-F_X(t))\big]\,\dd t\\
  &=& \EE[|X|] + \EE[|Y|].
\end{eqnarray*}
Also, if $X, Y \in \P_n$ then $\sigma \in \M_{n-1}$, since, using integration by parts (i.e., Tonelli's Theorem), $$k \int_{-\infty}^\infty n|x|^{n-1}\,\dd\sigma(x) = -\int_{-\infty}^\infty |x|^n\,\dd (F_Y-F_X)(x) = \EE[|X|^n]-\EE[|Y|^n] < \infty.$$

Hence Lemma~\ref{l.mattner} implies that there exists a probability measure $\eta \in \P_{n-1}$ such that $\sigma * \eta \in \P_{n-1}$. Let $Z$ be a random variable independent from $X$ and $Y$ with distribution $\eta$. The measure $\sigma$ is by definition absolutely continuous with density $\frac{1}{k}(F_Y - F_X)$. Therefore, $\sigma * \eta$ is absolutely continuous as well, and its density $s$ satisfies, almost everywhere,\footnote{See  \citet[257xe]{fremlin}.}
\begin{eqnarray*}
  ks(x) 
  &=& \int_{\RR} (F_Y(x-t) - F_X(x-t))\,\dd \eta(t) \\
  &=& \PP[Y\leq x-Z] - d\PP[X \leq x-Z]\\
  &=& F_{Y+Z}(x) - F_{X+Z}(x).
\end{eqnarray*}
Because $\sigma * \eta$ is a probability measure and $d > 0$, then $F_{Y+Z}(x) \geq F_{X+Z}(x)$ for almost every $x$. Since the cdfs are right-continuous, this implies $F_{Y+Z} \geq F_{X+Z}$. Furthermore, this inequality is strict somewhere, since the integral of $k(F_{Y+Z} - F_{X+Z})$ is equal to one. Therefore, $X + Z >_\fosd Y + Z$. This concludes the proof of Theorem~\ref{thm1}; we have furthermore demonstrated that when $X, Y \in \P_n$ then $Z$ can be taken to be in $\P_{n-1}$, for any $n \in \{1,2,\ldots,\infty\}$.

We end this section by showing that the converse of Theorem~\ref{thm1} is also true: if $\EE[X] \leq \EE[Y]$ then there does not exist a random variable $Z$ such that $X+Z >_{\fosd} Y+Z$. This is immediate if $Z$ has finite expectation, since, as is well known, $X+Z >_1 Y+Z$ would imply that $\EE[X+Z] > \EE[Y+Z]$. For general $Z$, note that the measure 
\[
    \sigma'(A) = \int_A F_Y(t) - F_X(t)\,\dd t
\]
satisfies $\sigma'(\RR) \leq 0$, and so for any probability measure $\eta$ it holds that $[\sigma' * \eta](\RR) \leq 0$. It thus follows from the calculation above that $F_{Y+Z}(x) - F_{X+Z}(x) < 0$ for some $x$ (except in the trivial case in which $X$ and $Y$ have the same distribution). Thus it is impossible that $X+Z >_{\fosd} Y+Z$.

%
%




\section{Proof of Theorem~\ref{thm.ssod}}
\label{sec:proof-2}
Let $X$ and $Y$ be random variables in $\P_n$ with $\EE[X] = \EE[Y]$ and $k= \frac{1}{2}(\Var[Y] - \Var[X])  > 0$. Let $F_X$ and $F_Y$ be the cumulative distribution functions of $X$ and $Y$. Define the signed measure $\sigma$ as
\[
    \sigma(A) = \frac{1}{k}\int_A \int_{-\infty}^t F_Y(u) - F_X(u)\,\dd u\,\dd t.
\]
By using the assumptions that $\EE[X] = \EE[Y]$, $\EE[X^2],\EE[Y^2] < \infty$ and applying integration by parts, it follows that $\int_\RR |\int_{-\infty}^t F_X(u) - F_Y(u)\,\dd u\,| \dd t < \infty$. See, for instance, \citep[Lemma 15.2.1]{rachev2013} or \citep[p.160]{rachev2008advanced}. Hence $\sigma$ is bounded.
We show that a calculation similar to the one used in Theorem~\ref{thm1} shows that $\sigma(\RR)=1$ and that $\sigma \in \M_{n-2}$. More generally, we claim that for every $k$,
\[
    \int_{-\infty}^\infty t^k \dd \sigma(t) = \frac{1}{(k+1)(k+2)}\int_{-\infty}^\infty u^{k+2} \, \dd (F_Y - F_X)(u).
\]
As is well known, integration by parts implies
\[
    \int_{-\infty}^t F_Y (u)\, \dd u = \EE[(t-Y)^+] \text{~~and~~} \int_{-\infty}^t F_X (u)\, \dd u = \EE[(t-X)^+].
\]
Hence, for every $n$ and $k$, 
\begin{eqnarray*}
    \int_{-n}^n t^k \int_{-\infty}^t F_Y(u) - F_X(u)\,\dd u\,\dd t &=& \int_{-n}^n t^k (\EE[(t-Y)^+] - \EE[(t-X)^+]) \,\dd t. \\
    &=& \int_{-n}^n t^k \int_{-\infty}^n (t-u)^+ \, \dd (F_Y - F_X)(u) \,\dd t. \\
    &=& \int_{-\infty}^n \int_{-n}^n t^k (t-u)^+ \, \dd t \, \dd (F_Y - F_X)(u) \\
    &=& \int_{-\infty}^n \int_u^n t^k (t-u) \, \dd t \, \dd (F_Y - F_X)(u) \\
    &=& \int_{-\infty}^n \left(\frac{t^{k+2}}{k+2} -u\frac{t^{k+1}}{k+1}\right) \Bigg| ^n_u \, \dd (F_Y - F_X)(u) \\
    &=& \int_{-\infty}^n \frac{n^{k+2}}{k+2} -u\frac{n^{k+1}}{k+1} + \frac{u^{k+2}}{(k+1)(k+2)}  \, \dd (F_Y - F_X)(u)
\end{eqnarray*}
For every $n$, the last integral is well defined and finite provided $X,Y \in \P_{k+2}$. It is a standard result that every $W \in \P_{k+2}$ satisfies $n^{k+2}[1 - F_W(n) + F_W(-n)] \to 0$ as $n \to \infty$.%
\footnote{This follows from
\[
    n^{k+2}(\PP[|W|^{k+2} \geq n^{k+2}]) \leq \EE[\ind{|W|^{k+2} \geq n^{k+2}}|W|^{k+2}]
\]
and the observation that since $W^{k+2}$ is integrable, the right hand side must converge to $0$ as $n \to \infty$.
}
Hence $n^{k+2}(F_Y - F_X)(n) \to 0$ as $n \to \infty$. To see that $n^{k+1} \int_{-\infty}^n u\,\dd(F_Y-F_X)(u)$ converges to $0$, notice that
\[
    \int_{-\infty}^n n^{k+1} u \, \dd (F_Y - F_X)(u) \leq \int_{-n}^n n^{k+2} \, \dd (F_Y + F_X)(u) + \int_{-\infty}^{-n} |u|^{k+2} \, \dd (F_Y + F_X)(u).
\]
The first term on the right hand side converges to $0$ by the same argument as above, and the second term converges to $0$ from the assumption that $X,Y \in \P_{n+2}$. This concludes the proof of the claim.

As in the proof of Theorem~\ref{thm1}, 
we can invoke Lemma \ref{l.mattner} to prove the existence of a probability measure $\eta \in \P_{n-1}$ such that $\sigma * \eta \in \P_{n-1}$. Let $Z$ be a random variable independent from $X$ and $Y$ with distribution $\eta$. Then $s(x)$, the probability density function of $\sigma * \eta$, is
\begin{eqnarray*}
  ks(x) 
  &=&\int_{\RR} \int_{-\infty}^{x-t} F_Y(u) - F_X(u)\,\dd u\,\dd \eta(t) \\
  &=&\int_{\RR} \int_{-\infty}^{x} F_Y(u-t) - F_X(u-t)\,\dd u\,\dd \eta(t) \\
  &=&\int_{-\infty}^{x} \int_{\RR}  F_Y(u-t) - F_X(u-t)\,\dd \eta(t)\,\dd u \\
  &=&\int_{-\infty}^{x} F_{Y+Z}(u) - F_{X+Z}(u)\,\dd u.
\end{eqnarray*}

Since $\sigma * \eta$  is a probability measure then $s(x)$ is non-negative for almost every $x \in \RR$. Since $F_{Y+Z}$ and $F_{X+Z}$ are right-continuous, this implies $s \geq 0$. Furthermore, this inequality is strict somewhere, since the integral of $s$ is equal to one. Therefore, $X + Z >_\sst Y + Z$, as $\int_{-\infty}^{x} F_{Y+Z}(u) - F_{X+Z}(u)\,\dd u \geq 0$ for all $x$ is a well known condition for second-order stochastic dominance \citep[see Theorem 4.A.2 in][]{shaked2007stochastic}.

\section{Uniformity}\label{s.uniformity}
\label{sec:uniform}

In this section we address the following two questions. First, is the noise term obtained in Theorem~\ref{thm1} robust to changes in the distribution of $X$ and $Y$? Moreover, can it be given a closed-form description? By restricting the attention to random variables with bounded support, we provide positive answers to both questions.

We consider pairs of random variables $X$ and $Y$ such that:
\begin{enumerate}[label=(\alph*)]
    \item Their support is included in a bounded interval $[-M,M]$; and
    \item Their difference in mean is bounded below by $\EE[X] - \EE[Y] \geq \varepsilon M$, where $\varepsilon > 0$.
\end{enumerate}

Given $M$ and $\varepsilon$, we construct a variable $Z$ that satisfies $X + Z >_\fosd Y + Z$ for \textit{any} pair $X$ and $Y$ for which (a) and (b) hold. In addition, we show that $Z$ can be taken to be a combination of uniformly distributed random variables.

The random variable $Z$ is defined by three parameters: $M$ and $\varepsilon$, as described above, as well as a parameter $a > 0$, which for the next result we are going to take to be sufficiently large. Let $U_1,U_2,\ldots$ be i.i.d.\ random variables that are uniformly distributed on the union of intervals $[-a-M,-a+M] \cup [a-M,a+M]$. Let $N$ be an independent geometric random variable with parameter $1/2$, so that $\PP[N=n]=2^{-1-n}$ for $n \in \{0,1,2,\ldots\}$. Finally, let $R_1$ and $R_2$ be variables independent from $N$ and $U_1,U_2,\ldots$ and uniformly distributed on $[-a,a]$. We define
\begin{equation}\label{eq.uniformity}
    Z = R_1 + R_2 + (U_1 + U_2 + \cdots U_N).
\end{equation}
So, the random variable $Z$ is obtained as a sum of mean-zero, independent noise terms that are uniformly distributed.

\begin{theorem}
\label{thm2}
  Fix $M, \varepsilon > 0$ and $a \geq 16M\varepsilon^{-1}+8M$. The random variable $Z$, as defined in (\ref{eq.uniformity}), is such that for every $X,Y$ supported in $[-M,M]$ with $\EE[X] - \EE[Y] \geq \varepsilon M$ it holds that
  \[
  X + Z >_\fosd Y + Z.
  \]
\end{theorem}

Notice that for smaller $\varepsilon$, as the difference in expectation between $X$ and $Y$ becomes negligible, the support of each term $U_i$ becomes increasingly large.
The random variable $Z$ is reasonably ``well-behaved'': for example, it has all moments, and exponentially vanishing tails. Using Wald's Lemma, its variance can be shown to be
\[
  \Var[Z] = \frac{2}{3} (M/\varepsilon)^2 \left(1024 +1024  \varepsilon+257 \varepsilon^{2}\right),
\]
so that its standard deviation is of order $M\varepsilon^{-1}$, and never more than $30 M\varepsilon^{-1}$.

To put this into perspective, consider an agent who must make a choice between two lotteries that pay between -\$10 and \$10, and whose expected value differs by at least \$1. Theorem~\ref{thm2} implies that there exists a zero mean independent background noise which---for {\em any} utility function---makes the lottery with the higher expectation preferable. Moreover, this noise need not be incredibly large: a standard deviation of \$3000 suffices.

\subsection{Proof of Theorem~\ref{thm2}}

To prove Theorem~\ref{thm2} we first prove a version of Lemma~\ref{l.mattner} that gives a stronger results for a smaller class of measures. Given $M,L > 0$, denote by $\M_M^L$ the set of bounded signed measures $\mu$ that are supported in $[-M,M]$, and for which $|\mu|([a,b]) \leq L(a-b)$ for all $a>b$ in $[-M,M]$.
\begin{proposition}\label{l.mattner1} 
  For every $M,L > 0$ there is a measure $\nu \in \P_\infty$ so that $\mu * \nu \in \P_\infty$ for every $\mu \in \M_M^L$ with $\mu(\RR)=1$.
\end{proposition}

Given $\mu \in \M_\infty$ and $a>0$, define 
\[
  \mu_a = (\mu-\delta) * \rho_a,
\]
where, as in the proof of Lemma~\ref{l.mattner}, $\delta$ is the point mass at $0$, and $\rho_a$ is the uniform distribution on $[-a,a]$. Let $r_a = \frac{1}{2a}\ind{[-a,a]}$ be the density of $\rho_a$. It follows that $\mu_a$ has a density $m_a(x) = \int_\RR r_a(x-t) \dd (\mu - \delta) (t)$ given by
\begin{equation}
\label{e.m_a}
  m_a(x) = \frac{1}{2a}\mu([x-a,x+a]) - r_a(x)
\end{equation}
and which satisfies the following properties:

\begin{lemma}
\label{c.m_a}
  For any $\mu \in \M_M^L$ it holds that
  \begin{enumerate}
      \item $|m_a(x)|$ is at most $M L/a+1/(2a)$.
      \item $m_a(x)$ vanishes outside of the union of the intervals $[-a-M,-a+M]$ and $[a-M,a+M]$.
  \end{enumerate}
\end{lemma}
\begin{proof}
\begin{enumerate}
  
  \item Since $\mu \in \M_M^L$, its norm $||\mu||$ is at most $2 M L$, and so $|\mu([x-a,x+a])| \leq 2 M L$. And since $|r_a| \leq 1/(2a)$, it follows that $|m_a| \leq M L/a + 1/(2a)$. 
  \item For $x$ such that $-a+M < x < a-M$ we have that $\mu([x-a,x+a])=1$, since $[x-a,x+a]$ includes $[-M,M]$ and thus all of the support of $\mu$. Since $r_a(x)=1/(2a)$ in this range it follows that $m_a(x)=0$. For $x<-a-M$ or $x>a+M$ we have that $r_a(x)=0$, and that likewise $\mu([x-a,x+a])=0$, since now $[x+a,x-a]$ does not intersect $[-M,M]$. Hence also in this range $m_a(x)=0$.
  
\end{enumerate}
\end{proof}

\begin{proof}[Proof of Proposition~\ref{l.mattner1}]
  The proof will follow the proof of Lemma~\ref{l.mattner} and will refer to some of the arguments given there. As in that proof, let $\rho_a$ be the uniform distribution on $[-a,a]$. Fix some $0 < c < 1$ and $\pi \in \P_\infty$, and let
  \[
    \tau = (1-c)\sum_{k=0}^\infty c^k \pi^{(k)}
  \]
  and
  \[
    \nu = \rho_a^{(2)} * \tau.
  \]
  Choose any $\mu \in \M_M^L$. Then as in the proof of Lemma~\ref{l.mattner} we have that
  \[
    \mu * \nu = \rho_a^{(2)} * (\mu -\delta + c \pi)*\tau + (1-c)\rho_a^{(2)}
  \]
  and that
  \begin{eqnarray*}
    \rho_a^{(2)} * (\mu -\delta + c \pi) 
    &\geq&  \frac{c}{2}  \rho_a * \left[\frac{2}{c} \rho_a * (\mu -\delta) +  \pi\right]\\
    &=& \frac{c}{2}  \rho_a * \left[\frac{2}{c}\mu_a +  \pi\right]. 
  \end{eqnarray*}
  We now show how to choose $a, c$ and $\pi$ so that $\frac{2}{c} \mu_a +  \pi$ is a positive measure, independent of our choice of $\mu \in \M_M^L$. By part 1 of Lemma~\ref{c.m_a} we know that the density of $\frac{2}{c} |\mu_a|$ is bounded from above by $2M L/( a c)+1/( a c)$, and by part 2 of the same lemma we know that it vanishes outside $[-a-M,-a+M] \cup [a-M,a+M]$. Therefore, if we choose $\pi$ to be the uniform distribution on this union of intervals then  $\frac{2}{c} \mu_a +  \pi$ will be positive if the density of $\pi$ on its support---which equals $1/(4M)$---is larger than our bound on the density of $\frac{2}{c}|\mu_a|$. That is, we would like $a$ and $c$ to satisfy
  \[
    \frac{1}{4M} \geq \frac{2 M L}{ a c} + \frac{1} { a c},
  \]    
  while keeping $c < 1$. Rearranging yields
  \[
    a c \geq 8 M^2 L + 4M,
  \]    
  which is satisfied if we take $c = 1/2$ and any $a \geq 16M^2 L + 8M$. The proof that $\nu \in \P_\infty$ is identical to the one in the proof of Lemma~\ref{l.mattner}.
\end{proof}

Given Proposition~\ref{l.mattner1}, the proof of Theorem~\ref{thm2} follows closely the proof of Theorem~\ref{thm1}. If $X$ and $Y$ are supported on $[-M,M]$ and if $\EE[X]-\EE[Y] \geq M\varepsilon$ then the measure $\sigma$, which is given by
\[
    \sigma(A) = \frac{1}{\EE[X] - \EE[Y]}\int_A F_Y(t) - F_X(t) \,\dd t,
\]
is in $\M_M^{1/(\varepsilon M)}$. Proposition~\ref{l.mattner1} shows we can find $\nu \in \P_\infty$ such that $\sigma * \nu \in \P_\infty$. In addition, as shown in the proof of the same proposition, given $a \geq 16M\varepsilon^{-1}+8M$, $\nu$ can be taken to be the distribution of the random variable $Z$ defined as (\ref{eq.uniformity}) in the main text. Finally, the same argument used in the proof of Theorem~\ref{thm1} shows that $X + Z >_\fosd Y + Z$.

\section{Calculations For Section \ref{sec:orders}}
\label{sec:calculations}
Let $X$ be a gamble paying $g > 0$ and $-l<0$ with equal probability. Let $Y = 0$. Following the proof Theorem \ref{thm1}, we first define $\mu$ to be the measure with support $[-l,g]$  and density
\[
    \frac{1}{\EE[X]}(F_Y - F_X)(s) = \begin{cases} 
          -\frac{1}{g-l} & s \in [-l,0] \\
          \frac{1}{g-l} & s \in [0,g] 
       \end{cases}
\]
and then seek to find a probability measure $\pi$, and parameters $a$ and $c$ such that

\begin{equation}\label{eq:ineq}
    (\rho_a * \mu -\rho_a) +  \frac{c}{\sqrt 2} \pi \geq 0.
\end{equation}

Let $\phi_a$ and $\Phi_a$ be, respectively, the pdf and CDF of a Gaussian random variable with mean zero and standard deviation $a$. Then $\rho_a * \mu -\rho_a$ has density
\[
    t(s) = \frac{1}{g-l} (2\Phi_a(s) - \Phi_a(s-g) - \Phi_a(s+l)) - \phi_a(s) .
\]

We claim that $t(s) \geq 0$ for $s > g$ and $t(s) \leq 0$ for $s < -l$. By the mean-value theorem, for each $s$ there exist $s_1 \in [s-g,s]$ and $s_2 \in [s, s+l]$ such that $\Phi_a(s) - \Phi_a(s-g) = \phi_a(s_1)g$ and $\Phi_a(s) - \Phi_a(s+l) = \phi_a(s_2)(-l)$. Hence
\begin{equation}\label{eq:ineq2}
    t(s) = \frac{1}{g-l}(\phi_a(s_1)g - \phi_a(s_2)l) - \phi_a(s).
\end{equation}
Suppose $s - g > 0$. Then $\phi_a(s_1)g > \phi_a(s)g$. Hence \eqref{eq:ineq2} is greater than
\[
    \frac{1}{g-l}(\phi_a(s)g - \phi_a(s_2)l) - \phi_a(s) = \frac{l}{g-l}(\phi_a(s) - \phi_a(s_2)).
\]
Since $0 < s < s_2$ then $\phi_a(s) - \phi_a(s_2) > 0$. This shows that $t(s) \geq 0$. Now suppose $s < -l$. Then $\phi_a(s_1)g < \phi_a(s)g$. Hence \eqref{eq:ineq2} is lower than
\[
    \frac{1}{g-l}(\phi_a(s)g - \phi_a(s_2)l) - \phi_a(s) = \frac{l}{g-l}(\phi_a(s) - \phi_a(s_2)),
\]
and because $s < s_2 < 0$ then $\phi_a(s) - \phi_a(s_2) < 0$. Therefore $t(s) \leq 0$. This concludes the proof of the claim.
We take $\pi$ to be absolutely continuous with density $f$. Then \eqref{eq:ineq} is satisfied if
\[
    f(s) \geq -\frac{\sqrt{2}}{c}t(s).
\]
We take $f(s) = 0$ if $s > g$ (since then $t(s) \leq 0$) and 
\[
    f(s) = \frac{\sqrt 2}{c} \left\vert t(s) \right\vert \text{~~if~~} s \leq g
\]
and then set
\[
    c = \sqrt 2 \int_{-\infty}^g \left\vert t(s) \right\vert \dd s
\]
so that $f$ integrates to $1$.

Given the parameters $g$, $l$ and $a$, the coefficient $c$ and the pdf $f$ can be computed numerically. This allows us to describe some quantitative features of the background noise $Z$, as discussed in section \ref{sec:orders}.

Table \ref{table:computation} provides some examples. For instance, when the binary gamble $X$ pays $\$12$ and -$\$10$ dollars with probability 1/2, there exists a background risk $Z = W + U_1 + \ldots + U_N$ such that $X + Z >_1 Z$ and: the Gaussian $W$ has standard deviation 3,500, $\PP[N=n]=(0.014)^n$ for $n \in \{0,1,2,\ldots\}$, and $Z$ has standard deviation 3,525.

\begin{table}
\begin{center}
\begin{tabular}{ r|r|r|r|r } 
   \multicolumn{1}{c}{$g$} & 
   \multicolumn{1}{|c}{$l$} &
   \multicolumn{1}{|c}{$\sigma_W$} & 
   \multicolumn{1}{|c}{$c$} & 
   \multicolumn{1}{|c}{$\sigma_Z$} \\
 \hline
 12     & 10    & 3,500         & 0.014      & 3,525           \\
 12     & 10    & 4,000         & 0.011      & 4,025          \\ 
 100    & 50    &  4,500        & 0.022      & 4,551          \\
 100    & 50    & 10,000        & 0.010      & 10,050         \\
 100    & 70    & 11,000        & 0.018      & 11,102         \\
\end{tabular}
\end{center}
\caption{Properties of the noise $Z = W + U_1 + \ldots + U_N$, where $W$ is Gaussian with mean zero and standard deviation $\sigma_W$, $(U_i)$ are i.i.d.\ random variables, $N$ is geometrically distributed with parameter $c$, and $\sigma_Z$ is the standard deviation of $Z$.
}
\label{table:computation}
\end{table}

\end{document}